\documentclass[a4paper,10pt]{article}
\usepackage[utf8]{inputenc}
\usepackage{amsmath, amsfonts, amsthm,amssymb, dsfont}
\usepackage{hyperref}
\usepackage{math} 
\usepackage[usenames, dvipsnames]{color}
\usepackage{geometry}

\newtheorem{lemma}{Lemma}[section]
\newtheorem{theorem}[lemma]{Theorem}
\newtheorem{proposition}[lemma]{Proposition}

\newtheorem{remark}[lemma]{Remark}

\newtheorem{definition}[lemma]{Definition}

\numberwithin{equation}{section}

\renewcommand{\bar}{\overline}
\newcommand{\h}{\hbar}

\newcommand{\Op}[1]{{\rm Op}^w_\hbar\left(#1\right)}

\title{Long time growth of  Sobolev norms in time dependent semiclassical  anharmonic oscillators 
}

\author{
E. Haus
\footnote{Dipartimento di Matematica e Fisica, Universit\`a degli Studi Roma Tre, Largo San Leonardo Murialdo, 00146, Roma, Italy
\newline
{\em Email: {\tt ehaus@mat.uniroma3.it}}},
 A. Maspero
 \footnote{
International School for Advanced Studies (SISSA), Via Bonomea 265, 34136, Trieste, Italy \newline
 \textit{Email: } \texttt{alberto.maspero@sissa.it} 
 }
}

\begin{document}

\maketitle

\begin{abstract}
We consider the semiclassical   Schr\"odinger equation on $\R^d$ given by   $$\im \h \partial_t \psi = \left(-\frac{\h^2}{2} \Delta  + W_l(x) \right)\psi + V(t,x)\psi ,
$$
 where $W_l$ is an anharmonic trapping of the form  $W_l(x)= \frac{1}{2l}\sum_{j=1}^d x_j^{2l}$, $l\geq 2$ is an integer and $\h$ is a semiclassical small parameter.
We construct a smooth potential $V(t,x)$, bounded in time with its derivatives,  and an initial datum such  that the Sobolev norms of the solution grow at a logarithmic speed for all times of order  $\log^{\frac12}(\h^{-1})$.
The proof relies on two ingredients: first we construct an unbounded  solution to a forced mechanical anharmonic oscillator, then we exploit 
semiclassical approximation with coherent states to obtain growth of Sobolev norms for the quantum system  which are valid for semiclassical time scales.
\end{abstract}

\section{Introduction and statement}
In this paper we  consider the semiclassical Schr\"odinger equation on $\R^d$, $d \geq 1$,  given by 
\begin{equation}
\label{sls}
\im \h \partial_t \psi = \left(-\frac{\h^2}{2} \Delta  + W_l(x) \right)\psi + V(t,x) \,  \psi \ , \qquad x\in \R^d \ ,
\end{equation}
where $W_l(x)$ is the anharmonic trapping  potential
$$
W_l(x) := \frac{1}{2l}\sum_{j=1}^d x_j^{2l} , \qquad l \in \N, \ \ \  l \geq 2 ,
$$
and   $\h \in (0,1]$ is a semiclassical parameter.
  We construct a time dependent perturbation 
\begin{equation}
\label{V}
V(t,x) := \beta(t)x_1  , 
\end{equation}
  with $\beta:\R \to \R$ smooth  and bounded with its derivatives,  so that \eqref{sls} has a solution  whose Sobolev-like norms grow  at a logarithmic speed
  for all times of order   $\log^{\frac12}(K \h^{-1})$, which is a scale  slightly shorter that the  Ehrenfest time.

The norms we use to measure the solution are the spectral ones associated with the 
anharmonic quantum oscillator
\begin{equation}
\label{aosc}
H_{l} := 1 - \h^2 \Delta  + W_l(x) \ .
\end{equation}
More precisely 
we define the scale of Hilbert spaces $\cH^r \equiv \cH^r(\R^d):= D(H_l^{r/2})$  (domain of $H^{r/2}_l$) for $r \geq 0$,  which we equip with the  Sobolev  norms\footnote{It turns out that such a space is equivalent to the 
space of functions
$$
\left\lbrace u \in L^2(\R^n) \colon \ \ \ 
\norm{(1- \h\Delta)^{r/2} u }_{L^2(\R^d)} + \norm{(1+|x|)^{rl} u}_{L^2(\R^d)} < + \infty \right\rbrace , 
$$
see e.g. \cite{YajimaZhang}.}
\begin{equation}
\label{norm}
\norm{ u}_{r}:=\norm{H_l^{r/2} \, u}_{L^2(\R^d)} < \infty, \qquad \forall r \geq 0.
\end{equation}
The negative spaces $\cH^{-r}$ are defined by duality with the $L^2(\R^d)$ scalar product. 
We also denote  $\cH^{\infty}:= \cap_r \cH^r$.

Our main result is the following one:
\begin{theorem}
\label{thm:main}
There exist $\psi_0 \in \cH^\infty$ and  a  function $\beta \in C^\infty(\R, \R)$ fulfilling
\begin{equation}
\label{beta}
\sup_{t \in \R }\abs{\partial_t^\ell \beta(t)} < + \infty , \  \ \ \ \forall \ell \in \N_0 ,
\end{equation}
such that the following is true. 
Denote by $\psi(t)$ the solution of  equation 
\eqref{sls} with $V(t,x) = \beta(t) x_1$ and initial datum $\psi_0$.
 Fix an arbitrary  $r \in \N$. Then there exist $\h_0, K_1, K_2, K_3 >0$ such that $\forall \h \in (0, \h_0]$ one has 
\begin{equation}
\label{est:main}
\norm{\psi(t)}_r \geq K_1 \,  \Big[\log (2+t)\Big]^{r} 
\end{equation}
for all times 
\begin{equation}
\label{times}
2 \leq t \leq K_2 \Big[\log\left( \frac{K_3}{\h}\right)\Big]^{\frac12} . 
\end{equation}
\end{theorem}

While in the last few years there has been lot of activity aiming to obtain  upper bounds on the growth of Sobolev norms \cite{nen, bourgain99, del, maro, BGMR2, Mon17, BertiMaspero}, there are only few results \cite{del2, BGMR1, ma18} which give lower bounds:   Theorem \ref{thm:main} goes in this direction, by exhibiting  a solution of \eqref{sls} whose norms increase for long but finite time. 

{The main difficulty in dealing with equation \eqref{sls} is that very few is known on the spectrum of the unperturbed operator $- \frac{\h^2}{2}\Delta + W_l(x)$. 
In particular we are not aware of any asymptotic expansion of its eigenvalues (a property that plays an important role in \cite{BGMR1, ma18}).}

In order to circumvent this problem, the idea  is to  exploit semiclassical approximation in a way that now we  briefly describe.
Equation \eqref{sls} with $V(t,x) = \beta(t) x_1$ is the quantization of the classical Hamiltonian 
 \begin{equation}
\label{clh}
H(t, q, p) = \frac{|p|^2}{2} + W_l(q) + \beta(t) q_1 , \qquad q, p \in \R^d , 
\end{equation}
whose equations of motion are given by
\begin{equation}
\label{mech}
\ddot{q}_1 + q_1^{2l-1} = \b(t)  ,  \qquad  \  \ddot{q_j} + q_j^{2l-1} = 0  , \ \  \ \forall \,   2 \leq j \leq d .
\end{equation}
We show, modyfing a construction of   Levi and Zehnder \cite{leze}, that it is possible to construct   $\beta\in C^\infty(\R, \R)$ bounded with all its derivatives  and an initial datum $(q_0, p_0) \in \R^{2d}$ such that the solution of \eqref{mech} with such an initial datum is unbounded; actually we show that the energy 
\begin{equation}
\label{def:mechen}
E(q,p) := \frac{|p|^2}{2} + W_l(q)
\end{equation}
along such a solution 
grows at a logarithmic speed as $t \to \infty$. \\
The next step  is to use the theory of semiclassical approximation with coherent states to convert  dynamical information on the mechanical system \eqref{mech} to the quantum system \eqref{sls}.
This is done in two steps. First we construct an approximate solution of \eqref{sls} using  coherent states.
A coherent state is a Gaussian packet which stays localized in the phase space along the trajectory of the mechanical system \eqref{mech} till the Ehrenfest time (see e.g. \cite{ hagedorn85, coro, BGP, BR, DeB}). 
As a consequence of the dynamics of \eqref{sls}, we are able to construct a coherent state which oscillates on longer and longer distances,  provoking a growth of its Sobolev norms. \\
The second step is to show that there exists a  solution of \eqref{sls} which stays close, in the $\cH^r$ topology,  to such   coherent state  for 
all times in  \eqref{times}.
This is done by extending  classical results of semiclassical approximation \cite{hagedorn85, coro} to the $\cH^r$ topology, a result which we think might be interesting in its own.\\


Theorem \ref{thm:main} extends partially to anharmonic oscillators a result of \cite{BGMR1}, which, in case of the quantum harmonic oscillators on $\R^d$, constructs solutions with unbounded path in Sobolev spaces.
More precisely, in \cite{BGMR1} it is proved that all the solutions of equation \eqref{sls} with $l=1$ (namely harmonic oscillators) and with 
\begin{equation}
\label{VBGMR}
V(t,x) := \frac{a}{2} \sin(t) x_1 , \qquad a \neq 0 
\end{equation}
have Sobolev norms growing at a polynomial speed:
\begin{equation}
\label{gBGMR}
\norm{\psi(t)}_r \geq C_r (1+t)^{2r} , \qquad \forall t \gg 1 . 
\end{equation}
Remark that,  in this case,  the growth of Sobolev norms happens  for all initial data, for all times and at a polynomial speed. 
The reason is that for system \eqref{sls} with $l=1$ and $V$ as in \eqref{VBGMR} the classical--semiclassical  correspondence is exact and valid for all times, a property first exploited by Enss and Veseli\'c \cite{enss83}.
This is also the mechanism exploited in \cite{BGMR1}, which ultimately  is based on the fact that    \eqref{mech} with $l = 1$ and  $\beta(t) = \sin t$ is a resonant system, whose solutions are  unbounded (see also \cite{del2, ma18} for different examples
of perturbations provoking growth of Sobolev norms).

In case $l \geq 2$, the classical-semiclassical correspondence is valid only for finite times, and  the speed of  growth of Sobolev norms is logarithmic and not polynomial in time. 
This is in accordance with the known upper bounds; 
in particular, in  dimension $d=1$, it is proved in \cite{BGMR2} that each solution of \eqref{sls}
grows at most subpolynomially in time, 
in the sense that   $\forall \epsilon , r >0$, there exists a constant
$C_{r,\epsilon}>0$ such that each solution of \eqref{sls}
fufills
\begin{equation}
\label{tep}
\norm{\psi(t)}_r \leq C_{r, \epsilon} \, (1+|t|)^\epsilon , \quad \forall |t| \geq 1 .
\end{equation}
If the map $t \mapsto \beta(t)$ is real analytic in time, the subpolynomial bound \eqref{tep} can be improved into a  logarithmic one \cite{maro}:
\begin{equation}
\label{tlog}
\norm{\psi(t)}_r \leq C_{r} \, \left[\log(1+|t|)\right]^{\frac{rl}{l-2}}, \quad \forall |t| \geq 1 .
\end{equation}
Remark that  Theorem \ref{thm:main} almost saturates the upper bound, at least for finite but long time intervals.
We are not aware of any results in which Sobolev norm explosion is achieved for all times.

In our opinion, our approach raises an interesting question: 
to which extent can dynamical properties of mechanical systems be converted into quantum analogous? 
Remark that mechanical systems of the form 
$\ddot q_1 + q_1^{2l-1}  = \beta(t) $
 or similar have been widely studied in the literature,
 and conditions on $\beta(t)$ are known  to guarantee either the boundedness of all solutions, or the  existence of unbounded ones,
 see e.g.  \cite{Littlewood66, leze, AlonsoOrtega98,   KKL, Morris,  DZ1, lale91,  levi91, Yuan98, WangYou16} and reference therein. 
For example if      $t \mapsto\beta(t)$ is periodic or quasi-periodic in time with a Diophantine frequency vector and $d = 1$, 
then each orbit of  \eqref{mech} in bounded \cite{Yuan98}. 

Before closing this introduction let us mention that the construction of unbounded orbits in  {\em nonlinear} Schr\"odinger equations
is an extremely difficult and challeging problem.
A first breakthrough was achieved in \cite{CKSTT}, which constructs solutions of the cubic nonlinear Schr\"odinger equation on $\T^2$ whose Sobolev norms become arbitrary large 
(see also \cite{hani14, guardia14, haus_procesi15, guardia_haus_procesi16, GHHMP}
 for generalizations of this result).
At the moment, existence of unbounded  orbits  has only been proved by G\'erard and Grellier \cite{gerard_grellier} for the cubic 
Szeg\H o equation on $\T$, and by 
Hani, Pausader, Tzvetkov and 
              Visciglia \cite{hani15} for the cubic NLS  on $\R \times \T^2$.

\medskip
\noindent{\bf Acknowledgments.} 
The authors thank D. Robert for many stimulating discussions and D. Bambusi for suggesting some references. 
During the preparation of this work we were partially supported by    Progetto  GNAMPA - INdAM 2018 ``Moti stabili ed instabili in equazioni di tipo Schr\"odinger''.

\section{Semiclassical pseudodifferential operators}
We recall the definition and main properties of  a class of semiclassical pseudodifferential operators adapted to study equation \eqref{sls}; the main reference for this part is   \cite{robook}.
   We start by denoting 
\begin{equation}
            \label{def:k0}
                        \tk_0(x,\xi) := (1+|x|^{2l}+|\xi|^{2})^{\frac{l+1}{2l }} \ , \qquad \forall x, \xi \in \R^{d} .
            \end{equation}     
The function $\tk_0$ is a good weight, in the sense that   there exists $C_l>0$ such that 
\begin{equation}
            \label{prop:k0}
 \tk_0(z+w) \leq C_l \,  \tk_0(z) \, \tk_0(w) , \qquad  \forall z, w \in \R^{2d} ,
            \end{equation}            
            and moreover 
 \begin{equation}
            \label{prop:k1}
\wt c_l(1+ E(x,\xi))^{\frac{l+1}{2l}} \leq   \tk_0(x,\xi) \leq \wt C_l (1+ E(x,\xi))^{\frac{l+1}{2l}} , 
            \end{equation}    
for some constants  $\wt c_l, \wt C_l >0$.
We begin with the following definition.
\begin{definition}
	A smooth function  $a(x, \xi) $ will be called a {\em symbol} in the class $\Sigma^m\equiv \Sigma^m(\R^{2d})$ if   $\forall \alpha, \beta \in \N^d$ there exists  $C_{\alpha \beta}>0$ such that 
	\begin{equation*}
	\abs{\partial_x^\alpha \partial_\xi^\beta a (x, \xi)} \leq  C_{\alpha \beta } \, \tk_0(x, \xi)^m  \ . 
	\end{equation*}
\end{definition}
Remark that we do not ask the derivatives of symbols to gain decay.
With this definition of symbols, one has  
$$x_j \in \Sigma^{\frac{1}{l+1}} , \quad \xi_j \in \Sigma^{\frac{l}{l+1}}, \quad
 |\xi|^{2} + W_l(x) \in \Sigma^{\frac{2l}{l+1}} , 
\quad 
\tk_0(x,\xi) \in \Sigma^{1} . $$
We endow $\Sigma^m$ with the family of semi-norms defined for any $M \in \N_0$ by
\begin{equation}
\label{seminorm}
\wp^m_M(a) := \sum_{|\alpha| + |\beta| \leq M} \ \sup_{x, \xi \in \R^{d}} \, \frac{\abs{\partial_x^\alpha \partial_\xi^\beta a (x, \xi)}}{\tk_0^{m}(x, \xi) }  \ . 
\end{equation}
As we already mentioned, we work with semiclassical operators, thus we  consider also symbols depending on the semiclassical parameter $\h \in (0,1]$. 
\begin{definition}
Let $a^\h$ be a family of symbols  depending smoothly on $\h\in]0, 1]$.  
We say that $a^\h\in \Sigma_u^m$ if $a^\h\in\Sigma^m$  for every $\h\in(0, 1]$ and if 
$$
\sup_{\h\in]0,  1]}\wp_{M}^m(a^\h) <+\infty, \qquad \forall M \in  \N^d_0.
$$
\end{definition}
Abusing notation, for a symbol $a^\h \in \Sigma^m_u$ we will denote by 
$\wp_M^m(a^\h)$ the seminorm \eqref{seminorm} where the supremum is taken also on $\h \in (0, 1]$.
To any symbol function $a^\h \in \Sigma^m_u$ we associate its $\hbar$-Weyl quantization $\Op{a^\h}$ by the rule
\begin{equation}
\Op{a^\h}[ \psi](x) := \frac{1}{(2 \pi\h)^d} \iint_{\R^{2d}}  e^{\frac{\im}{\hbar} ( x-y)\cdot  \xi } \ a^\h\left(\frac{x+y}{2}, \xi \right) \psi(y) \ \di y  \di \xi \ . 
\end{equation}
Sometimes we will write  $a(x, \h D_x)$ to denote the operator $\Op{a}$.

A classical result regards composition of pseudodifferential operators.
\begin{theorem}[Symbolic calculus]
Let $a^\h \in \Sigma^m_u$,  $b^\h \in \Sigma^{m'}_u$ be  symbols. Then  there exists a  symbol $c^\h \in \Sigma_u^{m+m'}$ such that $\Op{a^\h} \circ \Op{b^\h} = \Op{c^\h}$. 
For every $j \in \N$, there exists a positive constant $C$  and an integer $M \geq 1$ (both independent of $a^\h$ and $b^\h$) such that
$$
\wp_{j}^{m+m'} (c^\h) \leq C \, \wp_{M}^m(a^\h) \, \wp_M^{m'}(b^\h)\ .
$$
\end{theorem}


The second result concerns the boundedness of pseudodifferential operators.
\begin{theorem}[Calderon-Vaillancourt]
\label{thm:cv} 
Let $a^\h \in \Sigma^0_u$ be a symbol. Then $\Op{a^\h}$ extends to a linear bounded operator from $L^2(\R^d)$ to itself. 
Moreover there exist  constants $C, N >0$ such that 
\begin{equation}
\sup_{\h \in (0, 1]}\norm{\Op{a^\h}}_{\cL(L^2)} \leq C  \ \wp^0_N(a^\h) \ .
\end{equation}
\end{theorem}
An immediate consequence of Calderon-Vaillancourt theorem and symbolic calculus is that if  $a^\h \in \Sigma^m_u$, $m \in \R$, then 
 $\Op{a^\h}$ maps $\cH^{r+\frac{m(l+1)}{l}}$ to $\cH^{r}$ $\, \forall r \in \R$ with the quantitative bound
\begin{equation}
\label{CV2}
\sup_{\h \in (0, 1]} \norm{\Op{a^\h} }_{\cL\left(\cH^{r+\frac{m(l+1)}{l}}, \cH^{r}\right)} \leq C'  \ \wp_{N'}^m(a^\h)  ,
\end{equation}
where $C', N'$ are positive constants.

The next result is  the exact Egorov theorem.
\begin{proposition}[Exact Egorov]
	\label{thm:eq}
	Let $\chi(t, x, \xi)$ be a polynomial function in  $x, \xi$ of degree at most two with smooth $t$-dependent coefficients. Let $U_\chi^\h(t,s)$ be the propagator of  the Schr\"odinger equation $\im \h\dot \psi = \chi(t, x, \h D_x) \psi$.
	 Then for every  $a^\h \in \Sigma^m_u$, one has 
	$$
	U_\chi^\h(t,0)^* \, \Op{a^h} \, U_\chi^\h(t,0) = \Op{a_{t}^\h} ,
	 \qquad
	{a^\h_{t}} := a^\h\circ \phi^{t}_\chi , 
	$$
	where $\phi^t_\chi(x,\xi)$ is  classical  Hamiltonian flow at time $t$ of $\chi(t, x, \xi)$ with initial datum $(x,\xi)$ at time 0.
\end{proposition}
We denote by $\cT_\hbar(z)$ the Weyl operator
 \begin{equation}
\label{weyl.op}
\left[\cT_\hbar(z)\psi\right](x):= \left[\exp\left(- \frac{\im }{\hbar} (-p \cdot x + q \cdot \hbar D_x)\right)\psi\right](x) 
\end{equation}
remark that $\cT_\hbar(z)$ is the time 1 flow of the 
Schr\"odinger equation  $\im \h\dot \psi = \chi(z; x, \h D_x) \psi$, where
$z:=(p,q)\in \R^d\times $  and 
 $\chi(z; x, \xi):= -p\cdot x + q\cdot \xi$ is a linear Hamiltonian.
 By Proposition \ref{thm:eq}   one gets
\begin{equation}
\label{T.op}
\cT_\hbar(z)^* \, \Op{a} \, \cT_\hbar(z) = \Op{a_z} , 
\qquad a_z(x,\xi) := a(x+q, \xi + p) .
\end{equation}
We will also use the dilation operator
 $$
 \left[\Lambda_\hbar \psi\right](x) := \frac{1}{\hbar^{d/4}} \, \psi\left(\frac{x}{\sqrt \hbar}\right) \ , \qquad \mbox{for } \psi \in L^2(\R^d) ;
 $$
 it is  unitary on $L^2(\R^d)$ and 
 conjugates pseudodifferential operators in the following way:
\begin{equation}
\label{dil}
\Lambda_\hbar^{-1}\, \Op{a} \, \Lambda_\hbar = {\rm Op}^w_1(b) \ , \qquad b(x, \xi):= a\left(\sqrt \hbar x, \sqrt \hbar \xi\right) .
\end{equation}

\section{Semiclassical approximation and coherent states}
Consider the semiclassical  Schr\"odinger equation
\begin{equation}
\label{semi}
\im \hbar\partial_t \psi =  H(t, x, \hbar D_x) \psi , 
\end{equation}
where $ H(t, x, \h D_x) $ is the $\hbar$-Weyl quantization of a real valued Hamiltonian $H(t, x, \xi)$ with $x, \xi \in \R^d$.
Through all the section we will make the following assumptions on both the classical symbol $H(t, x, \xi)$ and its Weyl quantization $ H(t, x, \h D_x)$.

\begin{itemize}
\item[(H$_{\rm cl}$)] $H(t,x, \xi)$  is a $C^\infty$--function in every variable.  There exists $m >0$ such that   $\forall t \in [0, T]$ the function $H(t, \cdot) \in \Sigma^m$.
 Its   Hamiltonian  flow,  namely the solution $(x(t), \xi(t))$ of 
\begin{equation}
\label{ham}
\begin{cases}
 \dot x = \displaystyle{\frac{\partial H}{\partial \xi}(t, x, \xi)} \\
 \dot \xi = \displaystyle{- \frac{\partial H}{\partial x}(t, x, \xi)} 
 \end{cases} ,  \qquad x(0) = x_0, \quad \xi(0) = \xi_0 
\end{equation}
  exists for all $t \in [0, T]$  and any initial datum $(x_0, \xi_0) \in \R^{2d}$.
\item[(H$_{\rm qu}$)] The Schr\"odinger equation 
\eqref{semi} has  a unique propagator $\cU^\h(t,s)$, unitary in  $L^2(\R^d)$ and fulfilling the group property $\cU^\h(t,s) \cU^\h(s, \tau) = \cU^\h(t, \tau)$. 
The propagator $\cU^\h(t,s) $ is bounded as a map from $\cH^r$ to itself  $\forall r $; 
moreover  there exists $\mu>0$ and, for every  $ r > 0$, a constant $C_r >0$ such that 
\begin{equation}
\label{gsn}
\sup_{\h \in (0,1]}\norm{\cU^\h(t,s)}_{\cL(\cH^r)} \leq C_r (1+| t-s |)^{r \mu} .
\end{equation}
\end{itemize}

Remark that, in the case of equation \eqref{sls}, assumption $({\rm H}_{\rm cl})$ is easily checked, while 
assumption $({\rm H}_{\rm qu})$ follows by 
Theorem \ref{thm:maro}, which is a 
  semiclassical version of the abstract theorem of growth proved in \cite{maro}.\\

We will construct an  approximate solution of \eqref{semi}  using coherent states.
Roughly speaking,  a coherent state is  a Gaussian packet concentrated in the phase space around a point $z = (q,p)\in \R^{2d}$. 
The theory of semiclassical approximation states that,   if the initial datum of equation \eqref{semi} is a coherent state concentrated near $z_0 = (q_0, p_0)$,  then the
true solution of \eqref{sls} stays  close, up to the Ehrenfest time, to a coherent state concentrated near  the solution  $z_t = (q(t), p(t))$  of the Hamiltonian equations of  $H(t, q, p)$ with initial datum $z_0$.

To state rigorously this result we need  to introduce some notation.
Define  for $z = (q,p) \in \R^{2d}$ the 
functions
\begin{align}
\label{c1}
&\vf_0(x) := \frac{1}{(\pi \hbar)^{d/4}}\, e^{-\frac{ |x|^2}{2 \hbar}} ,  \quad x \in \R^d , \\
\label{c2}
&\vf_z := \cT_\hbar(z)\vf_0 . 
\end{align}
The function $\vf_z$ is called a {\em coherent state}; it is a Gaussian packet localized in the phase space around  the point  $z \in \R^{2d}$. It is normalized so that   $\norm{\vf_z}_{L^2(\R^d)} = 1$.

Denote by $z_t = (q(t), p(t)) \in \R^{2d}$ the solution of the Hamiltonian equations of   $H(t, q, p)$ with initial datum $z_0 \in \R^{2d}$; let 
$M_t$ be the $2d\times 2d$ Hessian of the Hamiltonian computed at the solution $z_t$, namely
\begin{equation}
\label{def:M}
M_t := \left.\left( \frac{\partial^2 H}{\partial z^2}\right)\right|_{z = z_t} . 
\end{equation}
We use $z_t$ and $M_t$ to define the quadratic Hamiltonian
\begin{equation}
\begin{aligned}
H_2(t, x, \xi)  & := H(t, z_t) + \la x-q(t),  \frac{\partial H}{\partial q}(t, z_t)\ra_{\R^d} + \la \xi - p(t),  \frac{\partial H}{\partial p}(t, z_t)\ra_{\R^d}  \\
&\qquad + \frac{1}{2} \la M_t \begin{pmatrix}x - q(t) \\ \xi -p(t)\end{pmatrix}, \ \begin{pmatrix}x -q(t) \\ \xi - p(t) \end{pmatrix} \ra_{\R^{2d}} ,
\end{aligned}
\end{equation}
which is nothing but the Taylor expansion of order 2 of the Hamiltonian $H(t,q,p)$ around  $z_t$.
Its $\h$-quantization  $H_2(t, x, \h D_x)$ generates a  unitary propagator $\cU_2^\h(t,s)$ in $L^2(\R^d)$.

We denote by  $F_t$ the  solution of
\begin{equation}
\label{Ft}
\dot F_t = J M_t F_t \ , \qquad F_0 = \uno ,
\end{equation}
where $J := \begin{pmatrix} 0 & \uno \\ - \uno & 0  \end{pmatrix}$ is the standard Poisson tensor.
\begin{lemma}
\label{egorov}
Let $a^\h \in \Sigma^m_u$, $m \in \R$. Then
$ \cU_2^\h(t,0)^* \, \Op{a^\h} \, \cU_2^\h(t,0)$ is a pseudodifferential operator with symbol $a_t^\h$ given by
$$
a_t^\h(\zeta) = a^\h\left( z_t + F_t [\zeta - z_0] \right) \ , \quad \zeta = (x, \xi) \in \R^{2d} .
$$
\end{lemma}
\begin{proof}
Since $H_2(t, x, \xi)$ is a quadratic polynomial in $x, \xi$, we can apply Proposition \ref{thm:eq} and get  $\cU_2^\h(t,0)^* \, \Op{a^\h} \, \cU_2^\h(t,0) = \Op{a^\h\circ \phi^t_{H_2}}$, 
where $\phi^t_{H_2}$ is the Hamiltonian flow  of $H_2(t, x, \xi)$.  
We compute explicitly such a flow. Thus let  $\zeta(t):= \phi^t_{H_2}(\zeta)$ be the solution of 
$$
\dot \zeta = J M_t \zeta + J \grad H(t,z_t) - J M_t z_t , \qquad \zeta(0) = \zeta \in \R^{2d}. 
$$
By Duhamel's formula we get
\begin{equation}
\label{sss}
\zeta(t) = F_t \zeta + F_t \int_0^t F_{-s} J \grad H(s, z_s)\, \di s -
 F_t \int_0^t F_{-s} JM_s z_s  \, \di s .
\end{equation}
Now use that $z_s$ is a solution of the Hamiltonian equations of $H(s,z)$ to write  $J \grad H(s,z_s) = \frac{d}{ds} z_s $;  integrating by parts we obtain 
\begin{align}
\notag
F_t \int_0^t F_{-s} J \grad H(s,z_s)\, \di s 
&= z_t - F_t z_0 - 
F_t \int_0^t  \left(\frac{d}{ds} F_{-s}\right) z_s \, \di s  \\
\label{ss}
& = z_t - F_t z_0 +
F_t \int_0^t  F_{-s} J M_s z_s \, \di s ,
\end{align}
where in the last inequality we used that
$$
 \frac{d}{ds} F_{-s} =  \frac{d}{ds} F_{s}^{-1} = 
 - F_{s}^{-1} \, \left(\frac{d}{ds} F_{s} \right)\,F_{s}^{-1} = - F_{-s} J M_s .
$$
Inserting \eqref{ss} into \eqref{sss} gives the result.
\end{proof}

Now fix $z_0 \in \R^{2d}$ and consider the solution of \eqref{semi} with initial datum the coherent state $\vf_{z_0}$ defined in \eqref{c2}.
The main result of the section is  that the quantum evolution $\cU^\h(t,0)\vf_{z_0}$ is well approximated by the dynamics of  $\cU_2^\h(t,0)\vf_{z_0}$ in the  topology of  $\cH^r$, $\forall r \geq 0$. 
This extend to higher Sobolev spaces the results of \cite{coro}. 
To state the theorem precisely, let us 
 introduce  for any $T \geq 0$
 the quantities 
$$|F|_T := \sup_{0 \leq t \leq T} |F_t|  , 
\qquad
\cE_T := \sup_{0 \leq t  \leq T} E(z_t)  ,
$$
where $E(z) \equiv E(q,p)$ is the anharmonic energy defined in \eqref{def:mechen}.
\begin{theorem}
\label{thm:rob}
Assume $({\rm H}_{\rm cl})$ and  $({\rm H}_{\rm qu})$.
Fix  $z_0 \in \R^{2d}$,     $r \geq 0$ and $\kappa \in (0,1]$.
Then there exists a constant $\Gamma >0$ such that for any $\h, T >0$ fulfilling 
\begin{equation}
\label{cond}
\sqrt{\hbar} |F|_T \leq \kappa , 
\end{equation}
one has 
\begin{equation}
\label{approx}
\norm{\cU^\h(t, 0) \vf_{z_0} - \cU_2^\h(t,0) \vf_{z_0}}_r \leq \Gamma \,  \hbar^{1/2} \, |	F|_T^{3}(1+T)^{\mu r + 1} \, (1+\cE_T)^{\frac{r}{2}} , \ \ \  \quad \forall 0 \leq t  \leq T .
\end{equation} 
\end{theorem}


%

\begin{proof}
One starts with Duhamel's formula
\begin{equation}
\label{duh}
\cU^\h(t,0)\vf_{z_0} - \cU_2^\h(t,0)\vf_{z_0} = \frac{1}{\im \hbar} \int_0^t \cU^\h(t,\tau) \, \left[ {H}(\tau, x, \h D_x) - H_2(\tau, x, \h D_x) \right] \, \cU_2^\h(\tau, 0)\vf_{z_0} \, \di \tau .
\end{equation}
Recall  that $H_2(t,x,\xi)$ is the Taylor expansion at order two of $H(t,x, \xi)$ around the path $z_t$, thus  
\begin{equation}
\label{H-H2}
H(t, z) - H_2(t, z) = R(t, z-z_t), \quad z = (x,\xi) \in \R^{2d}
\end{equation}
 where
\begin{equation}
\label{R0}
\begin{aligned}
R(t, \zeta) = \sum_{\substack{\nu \in \N_0^{2d} \\ |\nu| = 3}} R_\nu(t, \zeta ) \cdot \zeta^\nu \ , \qquad 
 R_\nu(t, \zeta ):=  \frac{1}{(\nu - 1)!} \int_0^1 H^{(\nu)}\left(t,z_t + \theta \zeta\right) (1-\theta)^2 \, \di \theta . 
\end{aligned}
\end{equation}
Since $H(t, \cdot) \in \Sigma^m$, one has $R(t, \zeta) \in \Sigma^{m + 3l/(l+1)}$. Quantizing \eqref{H-H2} we obtain
\begin{equation}
\label{r.f}
{H}(\tau, x, \h D_x) - H_2(\tau, x, \h D_x) = \Op{R(\tau, \zeta- z_t)} .
\end{equation}
Inserting \eqref{r.f} into \eqref{duh} and taking the $\cH^r$ norm,  we have that
\begin{align*}
\norm{\cU^\h(t,0)\vf_{z_0} - \cU_2^\h(t,0)\vf_{z_0} }_r &\leq 
\frac{1}{ \hbar} \int_0^t 
\norm{\cU^\h(t,\tau)}_{\cL(\cH^r)} \norm{\Op{R(\tau, \zeta- z_\tau)} \cU_2^\h(\tau,0) \vf_{z_0}}_r \, \di \tau \\
& \stackrel{\eqref{gsn}}{\leq} C_r \frac{(1+|t|)^{1+r\mu}}{\hbar} 
\sup_{0 \leq \tau \leq t}  \norm{\Op{R(\tau, \zeta- z_\tau)} \cU_2^\h(\tau,0) \vf_{z_0}}_r 
\end{align*}
To control the last term we proceed as following. First remark that $\cU_2^\h(\tau,0)$ and $\cT_\hbar(z_0)$ are isometry in $L^2(\R^d)$, so is  
\begin{equation}
\label{U}
U^\h(t) :=  \cU_2^\h(\tau,0) \cT_\hbar(z_0) .
\end{equation}
Then, exploiting \eqref{c2} and the identity
\begin{align}
\label{U1}
U^\h(t)^*  \,\Op{a} U^\h(t) =   \Op{b} , \qquad b(t,\zeta):= a(z_t + F_t\zeta) ,
\end{align}
which follows by \eqref{T.op} and Lemma \ref{egorov},   we obtain 
\begin{align}
\notag
\norm{\Op{R(\tau, \zeta- z_\tau)}  \cU_2^\h(\tau,0) \vf_{z_0}}_r  
& = \norm{ H_l^{r/2}\Op{R(\tau, \zeta- z_\tau)} U^\h(\tau) \vf_0}_0\\
\notag
& = \norm{\left(U^\h(\tau)^*  \,H_l^{r/2} U^\h(\tau)\right) \Big(U^\h(\tau)^*\Op{R(\tau, \zeta- z_\tau)} U^\h(\tau) \Big)\vf_0}_0\\
\label{3l}
&  =    \norm{ \Op{{\mathtt h}_r(z_\tau + F_\tau\zeta)}  \, \Op{R(\tau, F_\tau\zeta)} \vf_0}_0 . 
\end{align}
In the last line we denoted  by ${\mathtt h}_r \in \Sigma^{\frac{lr}{l+1}}$ the symbol of $H_l^{r/2}$,  i.e. $H_l^{r/2} = \Op{{\mathtt h}_r}$.
We are left with estimating \eqref{3l}.
Let ${\mathtt h}_r(z_t; \zeta) := {\mathtt h}_r(\zeta + z_t)$. 
By \eqref{dil}
$$
\Lambda_\h^{-1} \Op{{\mathtt h}_r(z_\tau; F_\tau \zeta)}  \, \Op{R(\tau, F_\tau \zeta)} \Lambda_\h =
{\rm Op}^w_1\left({{\mathtt h}_r(z_\tau; \sqrt \h F_\tau \zeta)} \right) \, {\rm Op}^w_1{R(\tau, \sqrt \h F_\tau \zeta)}  .
$$
Thus,
using that $\Lambda_\h $ is unitary in $L^2(\R^d)$ and  writing $\vf_0 = \Lambda_\hbar \vf$, where $\vf(x):= \frac{1}{\pi^{d/4}} e^{-|x|^2}$,   we get 
\begin{align*}
 \eqref{3l} & 
 = \norm{ {\rm Op}^w_1\left({{\mathtt h}_r(z_\tau; \sqrt \h F_\tau \zeta)} \right) \, {\rm Op}^w_1\left({R(\tau, \sqrt \h F_\tau \zeta)}\right)  \vf }_0 .
\end{align*}
Now remark that $\vf$ is a Schwartz function, so by Calderon-Vaillancourt theorem  there exist $C, N >0$ such that 
\begin{align*}
\eqref{3l} \leq C \, 
\wp_N^{lr/(l+1)}\left({\mathtt h}_r(z_t; \sqrt \h F_t \zeta)\right) \ 
\wp_N^{m+3l/(l+1)}\left(R(\tau,  \sqrt \h F_t \zeta)\right) \ 
 \norm{\vf}_{\bar r} , 
\end{align*}
where  $\bar r=r+3 + m (l+1)/{l}$ . 
We are left with estimating the seminorms of the symbols.
By assumption \eqref{cond} we have 
$$
\sqrt \h |F_t| \leq \kappa \leq 1 , \qquad \forall \, 0 \leq t \leq T ,
$$
therefore the seminorm of ${\mathtt h}_r$ is controlled by
\begin{align*}
\wp^{lr/(l+1)}_N \left(\mathtt h_{r}(z_t; \sqrt \h F_t \zeta)\right) \leq  \wp^{lr/(l+1)}_N(\mathtt h_{r}) \sup_{\zeta \in \R^{2d}} \left|{\frac{\tk_0(\zeta + z_t)}{\tk_0(\zeta)}}\right|^{lr/(l+1)} .
\end{align*}
By  \eqref{prop:k0} and \eqref{prop:k1}, for any $t \in [0,T]$  we bound 
\begin{align*}
\sup_{\zeta \in \R^{2d}} \left|{\frac{\tk_0(\zeta + z_t)}{\tk_0(\zeta)}}\right|^{lr/(l+1)}  \leq  C'_l \tk_0(z_t)^{lr/(l+1)}  \leq C'_l \wt C'_l (1+ E(z_t))^{\frac{r}{2}} \leq C (1+ \cE_T)^{\frac{r}{2}} .
\end{align*}
Thus we proved
\begin{equation}
\label{shr}
\wp^{lr/(l+1)}_N \left(\mathtt h_{r}(z_t; \sqrt \h F_t \zeta)\right) \leq C (1+ \cE_T)^{\frac{r}{2}} , \qquad \forall \, 0 \leq t \leq T . 
\end{equation}
Consider now the seminorm of $R(\tau, \sqrt h F_t \zeta)$.
Proceeding as above and 
using the definition of $R$ in \eqref{R0} we obtain
\begin{equation}
\label{sR}
\wp_N^{m+3l/(l+1)}\left(R(\tau,  \sqrt \h F_t \zeta)\right) \leq  C (\sqrt \h |F|_T)^3 , \qquad \forall \, 0 \leq t \leq T . 
\end{equation}
Combining all estimates we have
$$
\norm{\cU^\h(t, 0) \vf_{z_0} - \cU_2^\h(t,0) \vf_{z_0}}_r \leq 
\Gamma  \frac{(1+T)^{1+r\mu}}{\hbar} 
 (\sqrt \h |F|_T)^3 (1+ \cE_T)^{\frac{r}{2}} , \qquad \forall \,0 \leq t \leq T 
$$
which proves \eqref{approx}.
\end{proof}

Theorem \ref{thm:rob} tells that it is possible to approximate,  in the $\cH^r$ topology,  the  quantum dynamics of a coherent state with the approximate flow generated by  a quadratic Hamiltonian. In the next proposition we show that it is easy to compute the values of observables along the approximate flow.

\begin{proposition}
\label{prop:gr}
Assume $({\rm H}_{\rm cl})$ and  $({\rm H}_{\rm qu})$.
Fix  $z_0 \in \R^{2d}$ and $\kappa \in (0,1]$.
 Furthermore assume that $a \in \Sigma^\rho$, $\rho \geq 0$, fulfills the condition
\begin{equation}
\label{a.ass}
\abs{\partial_x^\alpha \partial_\xi^\beta a(x,\xi)} \leq C_{\alpha \beta}  \ \tk_0(x,\xi)^{\rho - \frac{l\beta + \alpha}{l+1}} , \quad \forall |\alpha|+|\beta| \leq  1 .
\end{equation}
  Then there exist a constant $\Gamma_1 >0$ and for any 
  $\h_0, T>0$ fulfilling 
\begin{equation}
\label{h0FT}
\sqrt \h_0 \, |F|_T \leq \kappa , 
\end{equation}
  a smooth function $b: (0, \h_0]\times [0, T] \to \R$ such that 
\begin{equation}
\label{A}
\langle \Op{a} \, \cU_2^\h(t,0)\vf_{z_0}, \,  \cU_2^\h(t,0)\vf_{z_0} \rangle  =  a(z_t)  + b(\h, t) , 
\end{equation}
and moreover 
\begin{equation}
\label{R}
\abs{b(\h, t)} \leq  \Gamma_1 \, \h^\frac12  \, |F_t|  \left( 1+ \cE_t \right)^{ \frac{(l+1)\rho - 1}{2l}} , \qquad
\forall 0 \leq t \leq T , \ \quad \forall \h \in (0, \h_0] .
\end{equation}
\end{proposition}
\begin{proof}
With $U^\h(t)$ defined in \eqref{U} and exploiting \eqref{U1} we get 
\begin{align*}
\la \Op{a} \cU_2^\h(t,0)\vf_{z_0}, \cU_2^\h(t,0) \vf_{z_0} \ra 
& = \la U^\h(t)^* \Op{a}  U^\h(t) \vf_0, \vf_0 \ra 
 = \la \Op{a(F_t \zeta + z_t)} \vf_0, \vf_0 \ra 
\end{align*}
To compute the last scalar product we proceed as following.
Denote by $\Psi$ the orthogonal projector on $\vf_0$, $\Psi u := \la u, \vf_0 \ra \vf_0$; 
it is a pseudodifferential operator with $\h$-Weyl symbol given by the  Wigner function
$$
\cW_{\vf_0}(x,\xi) = 2^d \,  e^{- \frac{|x|^2 + |\xi|^2}{\hbar}} ,
$$
see e.g. \cite{coro}.
Now remark that for any operator $A$ one has
\begin{align*}
\la A \vf_0, \vf_0 \ra & = \la A \Psi \vf_0, \vf_0 \ra = \sum_{j \geq 0} \la A \Psi \vf_j, \vf_j \ra   = {\rm tr}(A \Psi) , 
\end{align*}
where $\{\vf_j\}_{j \geq 0}$ is any orthonormal basis  of $L^2(\R^{d})$ that completes $\vf_0$.\\
If $ A= \Op{a} $ is a   pseudodifferential operator, 
trace formula (see  \cite[Proposition II--56]{robook}) 
assures that 
$$ {\rm tr}(A \Psi) = (2\pi\h)^{-d}\int_{\R^{2d}} a(\zeta) \cW_{\vf_0}(\zeta) \, \di \zeta . $$
In our case  we  obtain 
\begin{align}
\notag
\la  \Op{a(F_t \zeta + z_t)}  \vf_0, \vf_0 \ra & = \frac{1}{(\pi \hbar)^d} \int_{\R^{2d}} a(F_t \zeta + z_t) \, e^{- \frac{|\zeta|^2}{\hbar}} \di \zeta \\
\label{b1}
& = \pi^{-d} \int_{\R^{2d}} a(\sqrt \hbar F_t \zeta + z_t) \, e^{- {|\zeta|^2}} \di \zeta
\end{align}
Now we write
\begin{align}
\label{b.dec}
a\Big(\sqrt \hbar F_t \zeta + z_t\Big) = a(z_t) + \tb(\h, t, \zeta) \ , \qquad 
\tb(\h, t, \zeta) := a\Big(\sqrt \hbar F_t \zeta + z_t \Big) - a(z_t) .
\end{align}
Inserting  \eqref{b.dec} in \eqref{b1} gives
formula \eqref{A} with  $b(\h, t) := \pi^{-d} \int \tb(\h, t,\zeta) e^{-|\zeta|^2} \di \zeta$. 
We prove now \eqref{R}. By Lagrange mean value theorem and assumption \eqref{a.ass}  we get
\begin{align*}
\abs{\tb(\h, t, \zeta)} 
&\leq   C \h^\frac12 \, |F_t| \, \norm{\zeta} \, \sup_{0 \leq s \leq 1}\tk_0\Big(\hbar^{\frac12} F_t \zeta + s z_t\Big)^{\rho - \frac{1}{l+1}} \\
&\leq  C' \h^\frac12  \, |F_t| \,  \norm{\zeta} \,\tk_0\Big(\hbar^{\frac12} F_t\zeta \Big)^{\rho - \frac{1}{l+1}} \, \tk_0(z_t)^{\rho - \frac{1}{l+1}}
\end{align*}
Now insert the last estimate in \eqref{b1}, and use 
\eqref{h0FT} and the inequality \eqref{prop:k1} 
to obtain the claimed  result.
\end{proof}

\section{Application to anharmonic oscillators}
In this section we apply Theorem \ref{thm:rob} to construct a solution of equation \eqref{sls} whose Sobolev norms grow for long but finite time.

\subsection{Unbounded orbits for classical anharmonic oscillator}
The first step is to consider the mechanical system \eqref{mech} and construct a forcing term $\beta(t)$, smooth and bounded with its derivatives, 
  so that there exists at least one unbounded solution.

This is the content of the  next result.
\begin{proposition}
\label{thm:mech}
There exists  a smooth function $\b \in C^\infty(\R, \R)$ fulfilling \eqref{beta}, 
such that equation \eqref{mech} possesses an unbounded solution $q(t)$. Moreover there exist $ C_1, C_2 >0$ s.t.
\begin{equation}
 \label{mech.en2}
 C_1 \, \left[\log(2+t)\right]^2 \leq  E\left(q(t), p(t)\right) \leq  C_2 \, \left[\log(2+t)\right]^2  \qquad \forall t\geq 0  .
 \end{equation} 
\end{proposition}

To prove the result we  follow  the strategy of  \cite{leze}.
First remark that the dynamic of \eqref{mech} is decoupled  into one dimensional systems. 
Since   $q_j = p_j = 0$ $\ \forall j \geq 2$ is an invariant subspace, we take an initial datum with $q_j(0) = p_j(0) = 0$ $\, \forall j \geq 2$.
Then the   dynamics of \eqref{mech} becomes one dimensional and restricted to the variables $(q_1, p_1)$.

 The idea is to create $\b(t)$ by 
giving a particular solution $q(t)$ of \eqref{mech} a ``helping kick'' to the right direction each time the
solution passes through the interval $-1 \leq  q_1 \leq  1$ from left to right, and make $\b(t)= 0$
at all other times. With such a $\b(t)$  the energy along the solution will increase at each passage from $-1$ to $1$ while remaining constant between consecutive passages.

Furthermore it is important to weaken the ``kicks'' at every passage, otherwise the external force $\beta(t)$ will have some derivatives which are unbounded in $t$. 

In order to construct $\b(t)$ we use an auxiliary  nonlinear equation. 
First define $g_1$ and $g_2$ to be positive cut-off functions on $\R$ s.t.  
\begin{equation*}
g_1(y):= 
\begin{cases}
1 \ , \quad |y| \leq 1/2 \\
0 \ , \quad |y| \geq 1 
\end{cases}  \ , 
\qquad
g_2(y):= 
\begin{cases}
0 \ , \quad y \leq 0 \\
1 \ , \quad y  \geq 1 
\end{cases}  \ , \qquad y \in \R . 
\end{equation*}
Then consider the nonlinear equation 
\begin{equation}
\label{mech.a}
\ddot y + y^{2l-1} = f(y, \dot y) , \qquad y, \dot y, \ddot y \in \R
\end{equation}
with 
\begin{equation}
\label{def.f}
f(y, \dot y) := g_1(y) \, g_2(\dot y) \, e^{- \dot y} \ . 
\end{equation}
Abusing notation, we denote again by $E(y, \dot y)$ the mechanical energy 
$$
E(y, \dot y):= \frac{\dot{y}^2}{2} + \frac{y^{2l}}{2l} . 
$$
\begin{proposition}
\label{prop:e}
Consider equation \eqref{mech.a}. The solution with initial datum    $y(0) = \dot y(0) = 1$ is globally defined and unbounded. More precisely there exist $C_1, C_2 >0$ s.t. 
\begin{equation}
\label{esty}
C_1 \log(1+t)^2 \leq E(y(t), \dot y(t)) \leq C_2 \log(1+t)^2 \ , \qquad \forall t \geq 1 \ . 
\end{equation}
\end{proposition}
\begin{proof}
Along a solution of \eqref{mech.a} the function $E(t) \equiv E(y(t), \dot y(t))$ fulfills 
\begin{equation}
\label{en.der}
\frac{\di}{\di t} E(t) = \dot y \, f(y, \dot y) \geq 0 \ . 
\end{equation}
More precisely $\frac{\di}{\di t} E(t) >0 $ when $\abs{y(t)} < 1$ and $\dot y(t) >0$, otherwise $\frac{\di}{\di t} E(t) = 0 $.

Set $t_0=0$, define the increasing sequence of all times $0 < t_1 < t_2 < \ldots$ such that $y(t_n) =\pm 1$, $\dot y(t_n) > 0$ for $n\geq1$, and denote $E_n := E(t_{2n})$ $\, \forall n \geq 0$. It is easy to verify that such a sequence is well defined.

By \eqref{en.der} and the definition of $f(y, \dot y)$ we have that $E_n$ is monotone increasing and furthermore
\begin{equation}
\label{en.c}
E(t) = E_n \ , \qquad \forall t_{2n} \leq t \leq t_{2n+1} , \qquad \forall n \geq 0  \ .
\end{equation}
Observe that $E_n > \frac{1}{2l}$ for all $n\geq 0 $.
Using that
$$
E_{n} \leq E(t) \leq E_{n+1} \ , \qquad \forall \   t_{2n+1} \leq t \leq t_{2n+2} , 
$$
and $\abs{y } \leq 1$ one obtains the bound 
\begin{equation}
\label{doty}
\sqrt{2 E_n - \frac{1}{l}} \leq \dot y(t) \leq \sqrt{2 E_{n+1}} \ , \quad \forall \   t_{2n+1} \leq t \leq t_{2n+2} .
\end{equation}
Next write 
\begin{align*}
E_{n+1} - E_n = \int_{t_{2n}}^{t_{2n+2}} \dot y(t) \, f(y(t), \dot y(t)) \, \di t = \int_{-1}^1 g_1(y) g_2(\dot y) e^{-\dot y} \, \di y ;
\end{align*}
this integral can be estimated by \eqref{doty} obtaining
\begin{equation}
\label{en.est}
c e^{- \sqrt{2E_{n+1}}} \leq E_{n+1} - E_n \leq 2 e^{- \sqrt{2 E_n - \frac{1}{l}}}  \leq 2 e^{\sqrt{\frac{1}{l}}} e^{-\sqrt{2 E_n}}  \ , \qquad \forall n \geq 0 \ , 
\end{equation}
for some constant $c>0$ depending only on the choice of the cutoff functions $g_1, g_2$.
We claim that 
$$
\lim_{n \to \infty} E_n = + \infty \ . 
$$
Indeed, the limit exists since $\{ E_n \}_{n \geq 1}$ is an increasing sequence. Assuming that $\lim\limits_{n \to \infty} E_n = E_\infty < \infty$, one gets a contradiction when passing to the limit in \eqref{en.est} (recall that $E_n>\frac{1}{2l}$).

Now use  \eqref{en.est} and the fact that  $E_n \geq \frac{1}{2l}$ $\, \forall n$, to get that $1 \leq E_{n+1}/E_{n} \leq K := 1+4l$, which implies that
\begin{equation}
\label{en.est2}
c e^{- \sqrt{2 K E_{n}}} \leq E_{n+1} - E_n \leq 2 e^{\sqrt{\frac{1}{l}}} e^{-\sqrt{2 E_n}} \ , \qquad \forall n \geq 0 \ . 
\end{equation}
To estimate $E_n$ we define the interpolating function
$$
\eta(\theta) = (\theta-n) E_{n+1} + (1+n -\theta) E_n \ , \quad n \leq \theta \leq n+1 , 
$$
so that the right derivative $D_+$ of $\eta$ fulfills
\begin{equation}
\label{en.est4}
c e^{- \sqrt{2K \eta(\theta)}} \leq D_+ \eta(\theta) \leq 2 e^{\sqrt{\frac{1}{l}}} e^{- \sqrt{2 \eta(\theta)/K}}   \ , \qquad \forall \theta \geq 0 \ . 
\end{equation}
To estimate $\eta$ we will use the method of the super and sub solutions. 
In particular, for any 
$$K'  > 2 K = 2(1+4l)$$
there exists $c_{K'} >0$ so that 
\begin{equation}
\label{en.est5}
c_{K'} \, \sqrt{\eta(\theta)} \, e^{- \sqrt{{K'} \eta(\theta)}} \leq D_+ \eta(\theta) \leq 2 e^{\sqrt{\frac{1}{l}}} \sqrt{2l\eta(\theta)} \, e^{- \sqrt{ 2\eta(\theta)/K}}   \ , \qquad \forall \theta \geq 0 \ ,
\end{equation}
where we used also that $\eta(\theta) \geq \frac{1}{2l}$.
The solution of this differential inequality can be estimated by the supersolution and subsolution method: in particolar consider the differential equations 
\begin{equation*}
\xi'(\theta) = c_{K'}\,  \sqrt{\xi(\theta)} \, e^{- \sqrt{{K'} \xi(\theta)}} \ , \quad  \zeta'(\theta) =  2 e^{\sqrt{\frac{1}{l}}} \sqrt{2l\zeta(\theta)} \, e^{- \sqrt{2 \zeta(\theta)/K}}  \ 
\end{equation*}
and initial condition $\xi(0) = \zeta(0) = \eta(0)$. Then one has $\xi(\theta) \leq \eta(\theta) \leq \zeta(\theta)$ for all $\theta \geq 0$.
A simple computation shows that
\begin{equation*}
\frac{1}{{K'}} \left[\log \left(e^{ \sqrt{{K'}} \eta(0)} + \frac{c_{K'} \sqrt{{K'}}}{2} \theta\right) \right]^2 \leq \eta(\theta) \leq \frac{K}{2} \left[\log\left(e^{\sqrt{2\eta(0)/K}} + 2 \sqrt{l/K} e^{\sqrt{1/l}} \,  \theta\right) \right]^2   \ , \qquad \forall \theta \geq 0 \ . 
\end{equation*}
Evaluating this expression at $\theta = n$, one has 
\begin{equation}
\label{en.est6}
\ta [\log (2+ n)]^2  \leq E_n \leq \tb [\log( 2+n)]^2 \ , \qquad \forall n  \geq 0 \ 
\end{equation}
for some positive constants $\ta, \tb$.
Now we need to relate $n$ with $t_n$. To do so,
denote by $T(E)$ the period of oscillation of the solutions of $\ddot y  + y^{2l-1} = 0$ with energy $E >0$.
It is given by
\begin{equation}
\label{T.per}
T(E) = c_l \, E^{- \frac{l-1}{2l}}
\end{equation}
for some constant $c_l >0$.
Moreover in our case 
\begin{equation}
\label{period.est}
\frac{T(E_{n+1})}{2} \leq t_{2n+2} - t_{2n} \leq T(E_n) \ , \qquad \forall n \geq 0 . 
\end{equation}
Using \eqref{en.est6} and the explicit expression \eqref{T.per}, one has
\begin{equation}
\label{en.est7}
\ta\, [\log(2+n)]^{- \frac{l-1}{l}}    \leq t_{2n+2} - t_{2n} \leq  \tb \, [\log (2+n)]^{- \frac{l-1}{l}} \ , \qquad \forall n  \geq 0 \ . 
\end{equation}
for some new constants $\ta, \tb$ different from those in \eqref{en.est6}.
Now write 
$t_{2n} = \sum_{m = 0}^{n-1} t_{2m+2} - t_{2m}$, thus \eqref{en.est7} and the estimates  
$$
c_1 \frac{n}{[\log (2+n)]^\alpha} \leq \sum_{m=0}^{n-1} \frac{1}{[\log (2+m)]^\alpha} \leq c_2 \frac{n}{[\log (2+n)]^\alpha} , \qquad \forall \alpha \in [0,1] , \ \ \forall n \geq 1 
$$
show that 
\begin{equation}
\label{en.est8}
 \frac{\tilde \ta \, n}{[\log (2+n)]^{(l-1)/l}} \leq t_{2n} \leq  \frac{\tilde \tb \, n}{[\log(2+ n)]^{(l-1)/l}} , \qquad \forall n \geq 1 . 
\end{equation}
By \eqref{en.est8} and \eqref{en.est6}, we get that 
\begin{equation}
C_1 \leq  \frac{E_n}{[\log(2+ t_{2n})]^2} \leq C_2  \ , \qquad \forall n \geq 0 \ ,
\end{equation}
which implies \eqref{mech.en2}.
\end{proof}

\begin{proof}[Proof of Proposition \ref{thm:mech}]
Let $ y(t)$ be the solution of \eqref{mech.a} with initial datum $(y(0), \dot{ y}(0)) = (1,1) $.
Define   $\beta(t)$  as
\begin{equation}
\label{beta.def}
\beta(t) := f\left(y(t), \dot{y}(t)\right) \ .
\end{equation}
Then $q(t):=( y(t), 0, \ldots, 0)$, $\dot q(t):=( \dot y(t), 0, \ldots, 0)$
is the solution of 
\eqref{mech} with initial datum 
$q(0) = (1, 0, \ldots, 0)$ and $\dot q(0) = (1, 0, \ldots, 0)$.
By Proposition \ref{prop:e}, the energy along $q(t)$ increases, and \eqref{mech.en2} holds (remark that  $p(t) = \dot q(t)$).
To prove   \eqref{beta} it is enough to use Fa\`a di Bruno's formula and the fact that  $|y(t)| + |\dot y(t)| \to \infty$ as $t \to \infty$.
\end{proof}

\begin{remark}
Combining \eqref{esty} and   \eqref{beta.def}, one sees easily that $\beta(t) \to 0$ as $t \to \infty$. 
If instead $t \mapsto \beta(t)$ is periodic, it is known that  all the solutions of $\ddot q_1 + q_1^{2l-1} = \beta(t)$  are bounded in time  \cite{Morris,  DZ1, lale91,  levi91, Yuan98}. Remark that, for autonomous system,  the phenomenon of having all solutions bounded is very interesting and  often associated to some sort of integrability, for example as it happens in the defocusing cubic NLS  on $\T$ or the Toda lattice (see e.g. \cite{MasperoVey, bama}).
\end{remark}

Finally we need to  estimate the norm of $F_t$, 
which in this case is  defined  as the flow  of the linearized Hamiltonian \eqref{clh} 
 along the solution $(q(t), p(t))$ of Proposition \eqref{thm:mech}. 
By \eqref{Ft},   $F_t$ solves the equation 
\begin{equation}
 \label{Ft.a}
 \dot F_t = \begin{pmatrix}
0_d & -\Upsilon(t) \\
I_d & 0_d 
\end{pmatrix} F_t  \ , \qquad F_0 = \uno
 \end{equation}
where $I_d$ is the $d\times d$ identity matrix, 
$0_d$ the $d\times d$ zero matrix, 
and $\Upsilon(t)$ the $d\times d$ diagonal matrix defined by
$$ \Upsilon(t) := \diag \Big( (2l-2)(q_1(t))^{2l -2}, 0, \ldots, 0 \Big)  $$
where  $q_1(t) \equiv  y(t)$ is first component of the unbounded solution constructed in Proposition \ref{thm:mech}.
\begin{lemma}
\label{Ft.norm}
Consider equation \eqref{Ft.a}. There exists $\tc >0$ such that,   for all $T >0$, one has 
\begin{align}
\label{est.F}
 \sup_{0\leq t \leq T}  |F_t| \leq   \exp\Big( \tc T \left[\log(2+T)\right]^\varsigma\Big), \qquad \varsigma:= 2\left(1-\frac{1}{l}\right) .
\end{align}
\end{lemma}
\begin{proof}
Using the results of  Proposition \ref{thm:mech}, one gets
$$
|\Upsilon(t)| \leq  C_2 \,  (2l-2)  \left[\log(2+t)\right]^\varsigma ,  \quad \forall t \geq 1 , 
$$
therefore 
 $$
 |F_t| \leq \exp\left( \int_0^t (1 + |\Upsilon(s)|) \di s \right)
 \leq \exp\left( \tc t  \left[\log(2+t)\right]^\varsigma\right) ,
 $$
 which gives the thesis.
\end{proof}


\subsection{Growth of Sobolev norms}
In this section we apply the semiclassical approximation to the quantum Hamiltonian \eqref{sls}. 
The idea is that the coherent state stays localized in the phase space around  the solution $(q(t), p(t))$  of  \eqref{clh}, and therefore  oscillates more and more, increasing its Sobolev norms.
\begin{lemma}
\label{u2}
Let  $z_t:= (q(t), p(t))$ be the unbounded solution of Proposition \ref{thm:mech}, and denote by  $z_0 := (q(0), p(0))$ its initial datum.  
Fix an arbitrary $r \in \N$ and $0<\epsilon< 1$. Then there exist $\kappa, \h_0, \tC_1, \tC_2 >0$ such that $\forall \h \in (0, \h_0]$, one has 
\begin{equation}
\label{g}
\norm{\cU_2^\h(t,0)\vf_{z_0}}_{r} \geq  \tC_1 \left[\log (2+ t)\right]^r
\end{equation}
for all times
\begin{equation}
\label{}
2 \leq t \leq \tC_2 \left[\log \left( \frac{\kappa}{\sqrt \h}\right) \right]^{1-\epsilon} . 
\end{equation}
\end{lemma}
\begin{proof}
The result is an application of Proposition \ref{prop:gr}, which requires   the condition 
$\sqrt \h |F|_T \leq \kappa$ 
to be  fulfilled. 
Having fixed  $\kappa, \h_0 >0$ sufficiently small (to be specified later),  
and estimating $|F|_T$ by  
 Lemma \ref{Ft.norm},  we obtain that $T$ is constrained by the condition 
\begin{equation}
\label{T}
0 \leq T \leq \tC_2 \left[\log \left( \frac{\kappa}{\sqrt \h_0}\right) \right]^{1-\epsilon} , 
\end{equation}
where $\epsilon >0$ is an  arbitrarily small number and $\tC_2 \equiv \tC_2(\epsilon) >0$.
Define 
$$
\cE_r(x,\xi):= (E(x,\xi))^r \equiv  \Big(\frac{|\xi|^2}{2} + W_l(x) \Big)^r .
$$
The function $\cE_r$   is a symbol in $\Sigma^{2lr/(l+1)}$ fulfilling \eqref{a.ass}; moreover estimate  \eqref{CV2} implies that 
\begin{equation}
\label{est1}
\norm{\cE_r(x, \h D_x)^{1/2} \psi }_0 \leq C' \norm{\psi}_r , \quad \forall \psi \in \cH^r . 
\end{equation}
By Proposition \ref{prop:gr}  we have,  for every $t \in [0,T]$, the equality
\begin{equation}
\label{est2}
\la \cE_r(x, \h D_x)\, \cU_2^\h(t,0)\vf_{z_0}, \, \cU_2^\h(t,0)\vf_{z_0} \ra =  \cE_r(q(t), p(t)) + {b(\h, t)} ,
\end{equation}
where   $b(\h, t)$ fulfills,   by  \eqref{R} and \eqref{mech.en2}
\begin{equation}
\label{est4}
\abs{b(\h, t)} \leq  \Gamma_1  \h^\frac12  \, |F_t|  \left( 1+ E(z_t) \right)^{r - \frac{1}{2l}} \leq \Gamma_1 C_2^{r - \frac{1}{2l}} \kappa \Big[\log(2 + t)\Big]^{2r - \frac{1}{l}} ,  \qquad
\forall (\h, t)  \in (0, \h_0]\times [0,T] .
\end{equation}
The function $\cE_r(q(t), p(t))$ grows in time at a logarithmic speed; indeed by   Proposition \ref{thm:mech} 
\begin{equation}
\label{est3}
\cE_r(q(t), p(t))\geq C_1 \Big[\log(2 + t) \Big]^{2r}.
\end{equation}
Therefore collecting estimates  \eqref{est1}--\eqref{est4}
we obtain
\begin{align*}
\norm{\cU_2^\h(t,0)\vf_{z_0}}_r^2 & \geq \frac{1 }{C'^2} \left(C_1 \Big[\log(2+t)\Big]^{2r} - \Gamma_1 C_2^{r - \frac{1}{2l}}  \kappa  \Big[\log(2+t)\Big]^{2r-\frac1l} \right) \geq  \frac{C_1}{2 C'^2} \Big[ \log(2+t)\Big]^{2r} ,
\end{align*}
for all time $t$ provided 
$$
\te(\kappa)\leq t \leq T , \qquad \te(\kappa):=  \exp\left[\left(\frac{ 2\Gamma_1 C_2^{r - \frac{1}{2l}}  \kappa}{C_1}\right)^l\right] . 
$$
Now fix $\kappa>0$ so small  that $ \te(\kappa) \leq 2$, and $\h_0$ small enough so that $2$ is smaller than the r.h.s. of \eqref{T}.
\end{proof}

We can finally prove Theorem \ref{thm:main}.
\begin{proof}[Proof of Theorem \ref{thm:main}]
The result  is an application of  Theorem \ref{thm:rob} to system \ref{sls}. Assumption (H$_{\rm cl}$)  is trivially verified;
to show that  (H$_{\rm qu}$) holds note that by \eqref{CV2}
$$
\sup_{\substack{\h \in (0,1] \\ t \in \R}} \frac1\h \norm{[\beta(t) x_1, H_l] \psi }_{r} \leq C \norm{H_l^{\frac{1}{2}} \psi}_r  , \qquad \forall \psi \in \cH^{r+ 1} .
$$
Therefore condition \eqref{estT2} holds with $\tau = \frac12$ and Theorem \ref{thm:maro} implies that 
$$
\sup_{\h \in (0,1]} \norm{\cU^\h(t,s)}_{\cL(\cH^r)} \leq C_r' \la t-s \ra^{r} ;
$$
in particular condition \eqref{gsn} holds with $\mu = 1$.

Thus, by Theorem \ref{thm:rob}, Lemma \ref{u2} (with $\epsilon = 1/2$)  and using also \eqref{mech.en2},\eqref{est.F}, one finds constants $K_1, K_2, K_3>0$ such that 
\begin{align*}
\norm{\cU^\h(t,0)\vf_{z_0}}_r &\geq 
\norm{\cU^\h_2(t,0)\vf_{z_0}}_r - \norm{\cU^\h(t,0)\vf_{z_0} - \cU^\h_2(t,0)\vf_{z_0}}_r  \geq   
K_1 [\log(2+t)  ]^r
\end{align*}
for all times 
$$
2 \leq t \leq K_2 \left[\log\left( \frac{K_3}{\h}\right) \right]^{\frac12} . 
$$
\end{proof}

\appendix

\section{A semiclassical abstract theorem on growth of Sobolev norms}
\label{appA}

We prove here  a semiclassical version of Thereom  1.5 of \cite{maro}. Thus consider an Hilbert space $\cH^0$ and a positive, invertible, selfadjoint operator $K^\h$ (possibly $\h$-dependent) acting on it. 
Define the scale of spaces $\cH^r := D\left((K^\h)^{r}\right)$, endowed with the norm $\norm{\psi}_r \equiv \norm{(K^\h)^{r} \psi}_{\cH^0}$.
Note that the norms might depend on $\h$ as well.
On $\cH^r$, consider the time dependent  Schr\"odinger equation 
\begin{equation}
\label{eqL}
\im \hbar \partial_t \psi(t) = L^\hbar(t) \, \psi(t) \ , \qquad \psi\vert_{t=s} = \psi_s \in \cH^r
\end{equation}
where $L^\hbar(t)$ is a selfadjoint operator in $C^0\Big([0,T],\cL(\cH^{r+m}, \cH^{r})\Big)$, $m \in \R$.

\begin{theorem}
\label{thm:maro}
Assume that there exists  $\tau <1$  
such that the following holds true: $\forall r \geq 0$, there exists $C_r >0$  such that 
\begin{equation}
\label{estT}
\sup_{\hbar \in (0,1]} \frac{1}{\h} \, \norm{ \left[L^\hbar(t), K^\hbar\right] \, (K^\hbar)^{-\tau} }_{\cL(\cH^r)} \leq \,  C_r   , \quad  \forall t \in [0,T].
\end{equation}
Then equation \eqref{eqL} has a unique propagator $\cU^\h(t,s)$, $\forall t,s \in [0,T]$, unitary in $\cH^0$ which restricts to a bounded operator from $\cH^r $ to itself fulfilling
\begin{equation}
\label{estT2}
\sup_{\h \in (0,1]} \norm{\cU^\h(t,s)}_{\cL(\cH^r)} \leq C_r' \la t-s \ra^{\frac{r}{2(1-\tau)}} . 
\end{equation}
\end{theorem}
This result is proved in \cite{maro} for $\h =1$; here we prove its extension to the semiclassical case. 

\begin{proof}
The existence of the propagator, its unitarity in $\cH^0$ and the group property follows from   Theorem 1.5 of \cite{maro}. To prove \eqref{estT2} we revisit the proof of that theorem.
First by induction one verifies that  $\forall k \in \N$ 
\begin{equation}
\label{bb}
\sup_{\h \in (0,1]} \frac{1}{\h} \norm{\left[L^\h(t), (K^\h)^k\right] \, (K^\h)^{-(k-1+\tau)}}_{\cL(\cH^0)} \leq  C_k' .
\end{equation}
(see e.g. \cite[Lemma 2.1]{maro}).
Now remark that  $\cU^\h(t,s)$ is an isometry in $\cH^0$, so $\Vert\cU^\h(t,s)\psi_s\Vert_{k} = \Vert[\cU^\h(t,s)]^*\, (K^\h)^k \, \cU^\h(t,s)\psi_s\Vert_0.$
 But  we have
 $$
[\cU^\h(t,s)]^*\, (K^\h)^k \, \cU^\h(t,s)\psi_s = (K^\h)^k \psi_s + \frac{1}{\im \h}\int_s^t [\cU^\h(t_1,s)]^*\, [L^\h(t_1), (K^\h)^k]\, \cU^\h(t_1,s)\psi_s \, \di t_1
 $$
 Hence using \eqref{bb} we get the first estimate
\begin{equation}\label{step1}
 \Vert\cU^\h(t,s)\psi_s\Vert_{k} \leq  \Vert\psi_s\Vert_{k} + C_k'\int_s^t  \,  \Vert\cU^\h(t_1,s)\psi_s\Vert_{k-\theta}\di t_1,\quad \theta=1-\tau.
\end{equation}
After $m$  iterations of (\ref{step1}), with other constants $C_{k,m}$,   we get that $\Vert\cU^\h(t,s)\psi_s\Vert_{k} $ is bounded by 
\begin{align*}
C_{k,m} \sum_{j=0}^{m-1} \Vert\psi_s\Vert_{k - j\theta} 
\la t-s \ra^j  
+C_{k,m}\int_s^t\int_s^{t_1}\cdots \int_s^{t_{m-1}}\Vert\cU^\h(t_m,s)\psi_s\Vert_{k-m\theta} \ \di t_m \di t_{m-1}\cdots \di t_{1}.
\end{align*}
Now choose  $m$  so  that $m\theta \geq k$, thus $\Vert\cU^\h(t_m,s)\psi_s\Vert_{k-m\theta} \leq \Vert\cU^\h(t_m,s)\psi_s\Vert_{0} \leq \norm{\psi_s}_0$. 
If $\tau$ is rational, one can take $k$ sufficiently large so that $m = k/\theta$ is rational. Then one gets \eqref{estT2} with $r = k$. The general result follows from linear interpolation.
\end{proof}

\def\cprime{$'$}

\end{document}